\documentclass{gtmon_a}
\pdfoutput=1
\usepackage{pinlabel}

\proceedingstitle{Zieschang Gedenkschrift}
\conferencestart{5 September 2007}
\conferenceend{8 September 2007}
\conferencename{Conference in honour of Heiner Zieschang}
\conferencelocation{Toulouse, France}

\editor{Michel Boileau}
\givenname{Michel}
\surname{Boileau}

\editor{Martin Scharlemann}
\givenname{Martin}
\surname{Scharlemann}

\editor{Richard Weidmann}
\givenname{Richard}
\surname{Weidmann}

\title[The first Alexander {$\mathbb{Z}[\mathbb{Z}]$}--modules of
surface-links and of virtual links]{The first Alexander
$\mathbb{Z}[\mathbb{Z}]$--modules of surface-links\\and of virtual
links}

\author{Akio Kawauchi}
\givenname{Akio}
\surname{Kawauchi}

\address{Department of Mathematics\\Osaka City University\\\newline
Sugimoto\\Sumiyoshi-ku\\Osaka 558-8585\\Japan}
\email{kawauchi@sci.osaka-cu.ac.jp}
\urladdr{}

\volumenumber{14}
\issuenumber{}
\publicationyear{2008}
\papernumber{16}
\startpage{353}
\endpage{371}

\doi{}
\MR{}
\Zbl{}


\received{6 October 2005}
\revised{9 March 2007}
\accepted{9 March 2007}
\published{29 April 2008}
\publishedonline{29 April 2008}
\proposed{}
\seconded{}
\corresponding{}
\version{}

\arxivreference{}

\makeatletter
\def\cnewtheorem#1[#2]#3{\newtheorem{#1}{#3}[section]
\expandafter\let\csname c@#1\endcsname\c@Theorem}
\makeatother

\let\xysavmatrix\xymatrix
\def\xymatrix{\disablesubscriptcorrection\xysavmatrix}
\AtBeginDocument{\let\bar\wbar\let\tilde\wtilde\let\hat\what}

\newtheorem{Theorem}{Theorem}[section]
\cnewtheorem{Lemma}[Theorem]{Lemma}
\cnewtheorem{Sublemma}[Theorem]{Sublemma}
\cnewtheorem{Proposition}[Theorem]{Proposition}
\cnewtheorem{Corollary}[Theorem]{Corollary}
\cnewtheorem{Claim}[Theorem]{Claim}
\theoremstyle{remark}
\cnewtheorem{Conjecture}[Theorem]{Conjecture}
\cnewtheorem{Definition}[Theorem]{Definition}
\cnewtheorem{Example}[Theorem]{Example}
\cnewtheorem{Problem}[Theorem]{Problem}

\makeautorefname{Theorem}{Theorem}
\makeautorefname{Lemma}{Lemma}
\makeautorefname{Corollary}{Corollary}
\makeautorefname{Problem}{Problem}
\makeautorefname{Example}{Example}

\begin{document}

\begin{htmlabstract}
We characterize the first Alexander <b>Z</b>[<b>Z</b>]&ndash;modules
of ribbon surface-links in the 4&ndash;sphere fixing the number of
components and the total genus, and then the first Alexander
<b>Z</b>[<b>Z</b>]&ndash;modules of surface-links in the 4&ndash;sphere
fixing the number of components.  Using the result of ribbon
torus-links, we also characterize the first Alexander
<b>Z</b>[<b>Z</b>]&ndash;modules of virtual links fixing the number
of components.  For a general surface-link, an estimate of the total
genus is given in terms of the first Alexander
<b>Z</b>[<b>Z</b>]&ndash;module.  We show a graded structure on the
first Alexander <b>Z</b>[<b>Z</b>]&ndash;modules of all surface-links
and then a graded structure on the first Alexander
<b>Z</b>[<b>Z</b>]&ndash;modules of classical links, surface-links
and higher-dimensional manifold-links.
\end{htmlabstract}

\begin{abstract} 
We characterize the first Alexander $\mathbb{Z}[\mathbb{Z}]$--modules
of ribbon surface-links in the 4--sphere fixing the number of
components and the total genus, and then the first Alexander
$\mathbb{Z}[\mathbb{Z}]$--modules of surface-links in the 4--sphere
fixing the number of components.  Using the result of ribbon
torus-links, we also characterize the first Alexander
$\mathbb{Z}[\mathbb{Z}]$--modules of virtual links fixing the number
of components.  For a general surface-link, an estimate of the total
genus is given in terms of the first Alexander
$\mathbb{Z}[\mathbb{Z}]$--module.  We show a graded structure on the
first Alexander $\mathbb{Z}[\mathbb{Z}]$--modules of all surface-links
and then a graded structure on the first Alexander
$\mathbb{Z}[\mathbb{Z}]$--modules of classical links, surface-links
and higher-dimensional manifold-links.
\end{abstract}

\begin{asciiabstract}
We characterize the first Alexander Z[Z]-modules 
of ribbon surface-links in the 4-sphere fixing the number of components 
and the total genus, and then the first Alexander Z[Z]-modules 
of surface-links in the 4-sphere fixing the number of components. 
Using the result of ribbon torus-links,  we also characterize the first 
Alexander Z[Z]-modules of virtual links fixing the number of components. 
For  a general surface-link, an estimate of the total 
genus is given in terms of the first Alexander Z[Z]-module.  
We show  a graded structure on the first Alexander Z[Z]-modules 
of all surface-links and then a graded structure on 
the first Alexander Z[Z]-modules 
of classical links, surface-links and higher-dimensional manifold-links. 
\end{asciiabstract}

\maketitle

\section{The first Alexander $\Z[\Z]$--module of a surface-link}\label{sec1}

For every non-nagative partition $g=g_1+g_2+...+g_r$ of a non-negative
integer $g$, we consider a closed oriented 2--manifold $F=
F^r_g=F^r_{g_1,g_2,...,g_r}$ with $r$ components $F_i$ $(i=1,2,...,r)$
such that the genus $g(F_i)$ of $F_i$ is $g_i$. The integer $g$ is
called the {\it total genus} of $F$ and denoted by $g(F)$.  An
$F$--{\it link} $L$ is the ambient isotopy class of a locally-flatly
embedded image of $F$ into $S^4$, and for $r=1$ it is also called an
$F$--{\it knot}.  The {\it exterior} of $L$ is the compact 4--manifold
$E=S^4\backslash{\rm int}N(L)$, where $N(L)$ denotes the tubular
neighborhood of $L$ in $S^4$. Let $p\co \tilde E \to E$ be the infinite
cyclic covering associated with the epimorphism $\gamma\co  H_1(E)\to \Z$
sending every oriented meridian of $L$ in $H_1(E)$ to $1\in \Z$. An
$F$--link $L$ is {\it trivial} if $L$ is the boundary of the union of
disjoint handlebodies embedded locally-flatly in $S^4$.  A {\it
ribbon} $F$--{\it link} is an $F$--link obtained from a trivial
$F^r_0$--link by surgeries along embedded 1--handles in $S^4$ (see
Kawauchi, Shibuya and Suzuki \cite[page 52]{KSS}). When we put the
trivial $F^r_0$--link in the equatorial 3--sphere $S^3\subset S^4$, we
can replace the 1--handles by mutually disjoint $1$--handles embedded in
the 3--sphere $S^3$ without changing the ambient isotopy class of the
ribbon $F$--link by an argument of \cite[Lemma 4.11]{KSS} using a
result of Hosokawa and Kawauchi \cite[Lemma 1.4]{HK}.  Thus, every ribbon $F$--link is
described by a {\it disk--arc presentation} consisting of oriented
disks and arcs intersecting the interiors of the disks transversely in
$S^3$ (see \fullref{Link} for an illustration), where the oriented
disks and the arcs represent the oriented trivial 2--spheres and the
1--handles, respectively.

\begin{figure}[hbtp]
	\centering
	\includegraphics[clip,height=3cm]{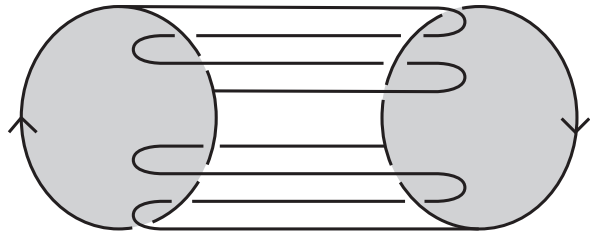}
	\caption{A ribbon $F^2_{1,1}$--link}
    \label{Link}
\end{figure}

Let $\Lambda= \Z[\Z] = \Z[t,t^{-1}]$ be  the integral Laurent 
polynomial ring. The homology $H_*(\tilde E)$ is a finitely 
generated $\Lambda$--module. Specially,  the first homology 
$H_1(\tilde E)$ is called the {\it first Alexander} $\Z[\Z]$--{\it module}, 
or simply 
the {\it module} of an $F$--link $L$ and denoted by $M(L)$. 
In this paper, we discuss the following problem:

\begin{Problem}\label{1.1}Characterize  the modules $M(L)$ of 
$F^r_g$--links  $L$ in a topologically meaningful class. 
\end{Problem}

In \fullref{sec2}, we discuss some homological properties of
$F^r_g$--links.  Fixing $r$ and $g$, we shall solve \fullref{1.1} for
the class of ribbon $F^r_g$--links in \fullref{sec3}.  We also solve
\fullref{1.1} for the class of all $F^r_g$--links not fixing $g$ as a
collorary of the ribbon case in \fullref{sec3}.  In \fullref{sec4}, we
characterize the first Alexander $\Z[\Z]$--modules of virtual links by
using the characterization of ribbon $F^r_{1,1,...,1}$--links.  In
\fullref{sec5}, we show a graded structure on the first Alexander
$\Z[\Z]$--modules of all $F^r_g$--links by establishing an estimate of
the total genus $g$ in terms of the first Alexander $\Z[\Z]$--module
of an $F^r_g$--link.  In fact, we show that there is the first
Alexander $\Z[\Z]$--module of an $F^r_{g+1}$--link which is not the
first Alexander $\Z[\Z]$--module of any $F^r_g$--link for every $r$
and $g$. In \fullref{sec6}, we show a graded structure on the first
Alexander $\Z[\Z]$--modules of classical links, surface-links and
higher-dimensional manifold-links.  We mention here that most results
of this paper are announced in \cite{KaAnnounce} without proofs.  A
group version of this paper is given in \cite{Ka''''}.

\section{Some homological properties on surface-links}\label{sec2}

The following computation on the homology $H_*(E)$ of the exterior 
$E$ of an $F^r_g$--link $L$ is done by using the Alexander duality for 
$(S^4,L)$:

\begin{Lemma}\label{2.1}
$$H_d(E)=\begin{cases}
\ \Z^{r-1} & \qquad(d=3)\\
\ \Z^{2g} &  \qquad(d=2)\\
\ \Z^r &  \qquad(d=1)\\
\ \Z &  \qquad(d=0)\\
\ 0 & \qquad(d\ne 0,1,2,3).
\end{cases}$$
\end{Lemma}

For a finitely generated $\Lambda$--module $M$, let $TM$ be 
the $\Lambda$--torsion part, and $BM=M/TM$ the $\Lambda$--torsion-free part. 
Let $\beta(M)$ be the $\Lambda$--rank of the module $M$, namely 
the $Q(\Lambda)$--dimension of the $Q(\Lambda)$--vector space 
$M\otimes_{\Lambda}Q(\Lambda)$, where $Q(\Lambda)$ denotes the quotient 
field of $\Lambda$.
Let 
\[ DM=\{ x\in M\mid \exists f_i\in\Lambda (i=1,2,...,s(\geqq 2))\,\,
\mbox{with \,\, $(f_1,...,f_s)=1$\,\, and \,\,  $f_i x=0$} \}, \] 
which is the maximal finite $\Lambda$--submodule of $M$ (cf Kawauchi \cite[Section 3]{Ka}), 
where the notation $(f_1,...,f_s)$ denotes the greatest common divisor of 
the Laurent polynomials $f_1,...,f_s$. We note that $DM$ contains all finite 
$\Lambda$--submodules of $M$, which is a consequence of $M$ being finitely 
generated over $\Lambda$. 
Let $T_DM=TM/DM$, and $E^qM=Ext^q_{\Lambda}(M,\Lambda)$.   
The following proposition is more or less  known 
(see J Levine \cite{Le} for  $S^n$--knot modules and  
\cite{Ka} in general):

\begin{Proposition}\label{2.2}
We have the following properties 
(1)--(5) on a finitely generated $\Lambda$--module $M$. 

\begin{enumerate}

\item $E^0M = {\rm hom}_{\Lambda}(M,\Lambda) = \Lambda^{\beta(M)}$, 

\item $E^1M=E^2M=0$ if and only if $M$ is $\Lambda$--free,

\item there are natural $\Lambda$--exact sequences 
$0\rightarrow E^1BM \rightarrow E^1M \rightarrow E^1TM \rightarrow 0$ and 
 $0\rightarrow BM\rightarrow E^0E^0BM \rightarrow E^2E^1BM\rightarrow 0,$

\item $E^1BM =DE^1M$,

\item $E^1TM={\rm hom}_{\Lambda}(TM,Q(\Lambda)/\Lambda)$ and 
$E^2M = E^2DM={\rm hom}_\Z(DM,\Q/\Z)$. 
\end{enumerate}
\end{Proposition}

The $d^{\text{\it th}}$\,$\Lambda$--{\it rank} of an $F^r_g$--link $L$
is the number $\beta_d(L)=\beta(H_d(\tilde E))$. We call the integer
$\tau(L)=r-1-\beta_1(L)$ the {\it torsion-corank} of $L$, which is
shown to be non-negative in \fullref{2.5}.  We use the following notion:

\begin{Definition}\label{2.3}
A finitely generated $\Lambda$--module $M$
is a {\it cokernel-free} $\Lambda$--module of {\it corank} $n$  
if there is an isomorphism 
$M/(t-1)M\cong \Z^n$ as abelian groups.
\end{Definition}

The corank of a cokernel-free $\Lambda$--module $M$ is denoted by $cr(M)$. 
We shall show in \fullref{3.3}  that  a $\Lambda$--module $M$  is 
a cokernel-free $\Lambda$--module of corank  $n$ if and only if 
there is an $F^{n+1}_g$--link $L$  for some $g$ such that $M(L)=M$. 
The following lemma implies that the cokernel-free $\Lambda$--modules   
appear naturally in the homology of an infinite cyclic covering:

\begin{Lemma}\label{2.4}
Let $p\co \tilde X\to X$ be an infinite cyclic covering 
over a finite  complex $X$. 
If $H_d(X)$ is free abelian, then  the $\Lambda$--modules 
$H_d(\tilde X)$, $TH_d(\tilde X)$ and $T_DH_d(\tilde X)$ are cokernel-free 
$\Lambda$--modules. In particular, if $H_1(X)\cong \Z^r$ and $\tilde X$ is 
connected, then $H_1(\tilde X)$ is cokernel-free of corank $r-1$. 
\end{Lemma}

\begin{proof}  
By  Wang exact sequence, the sequence  
$$H_d(\tilde X) \overset{t-1}{\rightarrow} H_d(\tilde X) 
\overset{p_*}{\rightarrow}H_d(X)
\overset{\partial}{\rightarrow}H_{d-1}(\tilde X)$$ 
is exact, which also induces an exact sequence
$$TH_d(\tilde X) \overset{t-1}{\rightarrow} TH_d(\tilde X) 
\overset{p_*}{\rightarrow}H_d(X),$$
for $(t-1)TH_d(\tilde X)=TH_d(\tilde X)\cap(t-1)H_d(\tilde X)$. Since $H_d(X)$ is free abelian, 
we have also the induced exact sequence 
$$T_DH_d(\tilde X) \overset{t-1}{\rightarrow} T_DH_d(\tilde X) 
\overset{p_*}{\rightarrow}H_d(X),$$ 
obtaining the desired result of the first half. 
The second half follows from the calculation that 
$$\text{im}[p_*\co  H_1(\tilde X)\to H_1(X)]
=\text{ker}[\partial\co  H_1(X)\to H_0(\tilde X)]\cong \Z^{r-1}.\proved$$
\end{proof}

From Lemmas \ref{2.1} and \ref{2.4}, we see that the $\Lambda$--modules  
$H_*(\tilde E)$, $TH_*(\tilde E)$ and\break $T_DH_*(\tilde E)$ are all 
cokernel-free $\Lambda$--modules for every $F^r_g$--link $L$. 
On these $\Lambda$--modules, we make the following calculations by using 
the dualities on the homology $H_*(\tilde E)$ in \cite{Ka}:

\begin{Lemma}\label{2.5} \ 

\begin{enumerate}
\item $\beta_1(L)=\beta_3(L)\leqq r-1$ and $\beta_2(L)=2(g-\tau(L))$, 

\item  $H_d(\tilde E)=0$ for $d\ne 0,1,2,3$, 
     $H_0(\tilde E)\cong \Lambda/(t-1)\Lambda$ and 
	 $H_3(\tilde E)\cong \Lambda^{\beta_1(L)}$,  

\item $cr(M(L))=r-1$ and $cr(TM(L))=cr(T_DM(L))=\tau(L)$, 
	
\item $cr(H_2(\tilde E))=2g-\tau(L)$ and 
	$cr(TH_2(\tilde E))=cr(T_DH_2(\tilde E))=\tau(L)$. 
\end{enumerate}
\end{Lemma}

\begin{proof}Since the covering $\partial \tilde E\to \partial E$ 
is equivalent to the product covering $F\times R\to F\times S^1$, we see that 
$H_*(\partial \tilde E)$ is a torsion $\Lambda$--module. Then  
the zeroth duality of \cite{Ka} implies $\beta_1(L)=\beta_3(L)$. 
The second duality of \cite{Ka} implies $E^2(H_3(\tilde E))=0$
 and the first duality of \cite{Ka} implies $E^1(H_3(\tilde E))=0$, 
 meaning that 
 $H_3(\tilde E)$ is a $\Lambda$--free module of $\Lambda$--rank $\beta_1(L)$. Since 
 $H_3(E)=\Z^{r-1}$, the Wang exact sequence implies $\beta_3(L)\leqq r-1$. 
 $H_k(\tilde E)=0$ for $k\ne 0,1,2,3$ and 
 $H_0(\tilde E)\cong \Lambda/(t-1)\Lambda$ are obvious. 
By \fullref{2.1}, the Euler characteristic $\chi(\tilde E;Q(\Lambda))$ 
of the $Q(\Lambda)$--homology $H_*(\tilde E;Q(\Lambda))$ is calculated 
as follows:  
$$\chi(\tilde E;Q(\Lambda))= -2\beta_1(L) + \beta_2(L) 
                           = \chi(E)= 2-\chi(F)= 2-2r+2g.$$
Hence we have $\beta_2(L)=2(g(F)-\tau(L))$, and (1) and (2) are 
proved. To see (3), 
the Wang exact sequence induces a short exact sequence 
\[ 0\rightarrow M(L)/(t-1)M(L)\rightarrow \Z^r\rightarrow \Z\rightarrow 0, \]
showing that $M(L)/(t-1)M(L)\cong \Z^{r-1}$ and $cr(M(L))=r-1$. 
Let $TM(L)/(t-1)TM(L)\cong \Z^s$ by \fullref{2.4}. 
Then we see that $BM(L)/(t-1)BM(L)$ has the $\Z$-rank $r-1-s$ 
by considering in the principal ideal domain 
$\Lambda_\Q=\Q[\Z]=\Q[t, t^{-1}]$ (although it may have a non-trivial integral 
torsion). This $\Z$--rank is also equal to $\beta_1(L)$, 
because $BM_\Q=BM\otimes_{\Lambda}\Lambda_\Q\cong \Lambda_\Q^{\beta_1(L)}$ and 
hence $BM_\Q/(t-1)BM_\Q\cong \Q^{\beta_1(L)}$. Thus, 
$$cr(TM(L))=s=r-1-\beta_1(L)=\tau(L).$$
Since $cr(TM(L))=cr(T_DM(L))$ is obvious, we have (3).  
To see (4), let 
$H_2(\tilde E)/(t-1)H_2(\tilde E)\cong \Z^u$ by \fullref{2.4}. 
Since $H_2(E)=\Z^{2g}$ by \fullref{2.1}, 
the kernel of $t-1\co TH_1(\tilde E)\to TH_1(\tilde E)$ has the 
$\Z$--rank $2g-u$, which is equal to the $\Z$--rank $\tau(L)$ of the cokernel of 
$t-1\co TH_1(\tilde E)\to TH_1(\tilde E)$ by considering it over $\Lambda_\Q$. 
Thus, $cr(H_2(\tilde E))=u=2g-\tau(L)$. Next, let 
$TH_2(\tilde E)/(t-1)TH_2(\tilde E)\cong \Z^v$. Then  
$BH_2(\tilde E)/(t-1)BH_2(\tilde E)$ has the $\Z$--rank $u-v$. Since 
$\beta_2(L)=2(g-\tau(L))$, we have $u-v=2(g-\tau(L))$ and 
$cr(TH_2(\tilde E))=v=\tau(L)$. Since 
$cr(TH_2(\tilde E))=cr(T_DH_2(\tilde E))$, we have (4). \end{proof}

The following corollary follows directly from \fullref{2.5}.

\begin{Corollary}\label{2.6}
An $F^r_g$--link $L$ has $\beta_*(L)=0$ 
if and only if $\beta_1(L)=0$ and $g=r-1$. 
\end{Corollary}

\section{Characterizing the first Alexander $\Z[\Z]$--modules 
of ribbon surface-links}\label{sec3}

For a finitely generated $\Lambda$--module $M$, let  
$e(M)$ be  the minimal number of $\Lambda$--generators of $M$. The following 
estimate is given by Sekine \cite{Se} and Kawauchi \cite{Ka'} for the case $r=1$ where 
we have $\tau(L)=0$: 

\begin{Lemma}\label{3.1}
If $L$ is a ribbon $F^r_g$--link, then we have 
\[g \geqq e(E^2M(L))+\tau(L).\]
\end{Lemma}

\begin{proof}Since $L$ is a ribbon $F^r_g$--link, 
there is a connected Seifert 
hypersurface $V$ for $L$ such that $H_1(V)$ and $H_1(V,\partial V)$ are 
torsion-free. In fact, we can take $V$ to be a connected sum of 
$r$ handlebodies and some copies, say $n$ copies, of $S^1\times S^2$ (cf \cite{KSS}). Then 
we have $H_1(V)=\Z^{n+g}$ and $H_2(V)=\Z^{n+r-1}$.  
Let $E'$ be the compact 4--manifold obtained from $E$ by splitting it along 
$V$. Let $\tilde V$ and $\tilde E'$ be the  lifts of $V$ and $E'$ by the 
infinite cyclic covering $p\co \tilde E\to E$, respectively. By 
the Mayer-Vietoris exact sequence, we have the 
following exact sequence 
\[ 0\rightarrow B \rightarrow H_1(\tilde V)\rightarrow H_1(\tilde E')
\rightarrow H_1(\tilde E)\rightarrow 0,\]
where $B$ denotes the image of the boundary operator 
$\tilde\partial\co  H_2(\tilde E)\to H_1(\tilde V)$. 
Since $H_1(V)\cong \Z^{n+g}$, we have $H_1(\tilde V)\cong \Lambda^{n+g}$. 
We note that 
$$H_1(E')\cong H_1(S^4-V)\cong H_2(S^4,S^4-V)\cong H^2(V)\cong \Z^{n+r-1},$$
so that $H_1(\tilde E')\cong \Lambda^{n+r-1}$. Using that $\Lambda$ has the 
graded dimension $2$, we see that $B$ must be a free $\Lambda$--module whose 
$\Lambda$--rank is calculated from the exact sequence to be 
\[(n+g)-(n+r-1-\beta_1(L))=g-\tau(L). \] 
Since by definition  
$E^2M(L)=E^2H_1(\tilde E)$ is a quotient $\Lambda$--module of 
$E^0B\cong \Lambda^{g-\tau(L)}$, we have $e(E^2M(L))\leqq g-\tau(L)$. 
\end{proof}

The following theorem is our first theorem, which shows that the estimate of 
\fullref{3.1} is  best possible and generalizes \cite[Theorem 1.1]{Ka'}.

\begin{Theorem}\label{3.2}
A finitely generated $\Lambda$--module 
$M$ is the module $M(L)$ of a ribbon $F^r_g$--link $L$ if and only if 
$M$ is a cokernel-free $\Lambda$--module of corank $r-1$ and 
$g\geqq e(E^2M)+\tau(M)$.
Further, if a non-negative partition $g=g_1+g_2 +...+g_r$ 
is arbitrarily given, then we  can take a ribbon 
$F^r_g$--link $L$ with  $g(F_i) = g_i$ for all  $i$.
\end{Theorem}

\begin{proof}
The \lq\lq only if\rq\rq part is proved by Lemmas \ref{2.5}
and \ref{3.1}. We show the \lq\lq if\rq\rq part. 
Let $M/(t-1)M\cong \Z^n$. We construct a ribbon 
$F^{n+1}_g$--link $L$ with $M(L)=M$ and $g=e(E^2M)+\tau(M)$ and observe that 
the module $M(L)$ is independent of a choice of the 
partitions $g=g_1+g_2 +...+g_r$ in our construction. 
This will complete the proof, since  an 
$F^{n+1}_{g'}$--link $L'$ with $g'>g$ and $M(L')=M$  can be obtained from 
$L$ by taking suitable connected sums of $L$ with  $g'-g$ trivial 
$F^1_1$--knots. The proof will be done by establishing the following three steps: 

\begin{enumerate}

\item Finding a nice $\Lambda$--presentation matrix $B$ for $M$.

\item Constructing a finitely presented group $G$ and an epimorphism 
$\gamma\co G\to \Z$ which induces a $\Lambda$--isomorphism 
${\rm Ker}\gamma/[{\rm Ker}\gamma,{\rm Ker}\gamma]\cong M$. 

\item Applying T. Yajima's construction to find a ribbon 
$F^r_g$--link $L$ with a prescribed disk--arc presentation such that 
$\pi_1(S^4\backslash L)=G$. 
\end{enumerate}

In (2), recall that ${\rm Ker}\gamma/[{\rm Ker}\gamma,{\rm Ker}\gamma]$ has 
a natural $\Lambda$--module structure with  
the $t$--action meant by the conjugation 
of any element $g\in G$ with $\gamma(g)=1\in \Z$. This $\Lambda$--module is  calculable 
from the group presentation 
of $G$ by the Fox calculus (see Kawauchi \cite{Kaw'} and H Zieschang \cite{Zi}). We shall show how 
to construct a desired Wirtinger presented group $G$ from the $\Lambda$--presentation 
$B$ of $M$ by this inverse process, so that we can establish (3). 
Let $m=e(E^2M)$ and $\beta=\beta(M)$. We take a $\Lambda$--exact sequence
$$0\rightarrow \Lambda^k\rightarrow\Lambda^{m+k}\rightarrow\Lambda^m 
\rightarrow E^2M\rightarrow 0$$ 
for some $k\geqq 0$, which induces a $\Lambda$--exact sequence
\[0\rightarrow \Lambda^m\rightarrow\Lambda^{m+k}\rightarrow\Lambda^k 
\rightarrow E^2E^2M=DM\rightarrow 0.\]
On the other hand, using $D(M/DM)=0$, we have 
$E^2(M/DM)=0$ and hence we have a $\Lambda$--exact sequence
\[0\rightarrow \Lambda^s\rightarrow\Lambda^{s+\beta} \rightarrow M/DM
\rightarrow 0\]
for some $s\geqq 0$. Thus, we have a $\Lambda$--exact sequence
\[0\rightarrow \Lambda^m\rightarrow\Lambda^{m+k+s}\rightarrow
\Lambda^{k+s+\beta}
\rightarrow M\rightarrow 0.\]
Let $B=(b_{ij})$ be a $\Lambda$--matrix of size $(k+s+\beta,m+k+s)$  
representing the $\Lambda$--homomorphism
$\Lambda^{m+k+s} \to \Lambda^{k+s+\beta}$. 
Since $M/(t-1)M=\Z^n$, we can assume 
\[B(1)=\left(
\begin{array}{cc}
              E^u&  O_{12}\\
               O_{21}&  O_{22}
               \end{array}\right)\]              
by base changes of $\Lambda^{m+k+s}$ and $\Lambda^{k+s+\beta}$, 
where $E^u$  is the unit matrix of size $u=k+s+\beta-n$, and $O_{12}, 
O_{21}, O_{22}$ are the zero matrices of sizes $(u,m-\beta+n), 
(n,u), (n,m-\beta+n)$, respectively. 
Let $b_{0j} =-\Sigma_{i=1}^{k+s+\beta} b_{ij}$, and $B^+ =(b_{ij})$  
$(0\leqq i \leqq k+s+\beta,1\leqq j \leqq m+k+s)$ 
We take $c_{ij}\in \Lambda$ so that 
\begin{equation} 
b_{ij}=\begin{cases}
(t-1)c_{ij} & \qquad(j>u)\\
(t-1)c_{ij}+\delta_{ij} &  \qquad(i>0,1\leqq j\leqq u)\\
(t-1)c_{ij}-1 & \qquad(i=0,1\leqq j\leqq u)
\end{cases}
\notag
\end{equation}  
Let $\gamma$ be the epimorphism from the free group 
$G_0=<x_0,x_1,...,x_{k+s+\beta}>$ onto $\Z$ defined by $\gamma(x_i)=1$, 
and $\gamma^+\co \Z[G_0]\to \Z[\Z]=\Lambda$ the group ring 
extension of $\gamma$ with $\gamma^+(x_i)=t$. Using that 
$\Sigma_{i=0}^{k+s+\beta} c_{ij}=0$, an algorithm of A Pizer \cite{Pa} 
enables us to find a word $w_j$ in $G_0$ such that $\gamma(w_j)=0$ and 
the Fox derivative 
\[\gamma^+(\partial w_j/\partial x_i)=c_{ij} (j=1,...,m+k+s)\] 
for every $i$. Let 
\begin{equation} 
R_j=\begin{cases}
x_jw_jx_0^{-1}w_j^{-1}& \qquad(1\leqq j \leqq u)\\
x_hw_jx_h^{-1}w_j^{-1}&  \qquad(u+1\leqq j \leqq m+k+s), 
\end{cases}
\notag
\end{equation}  
where we can take any $h$ for the $x_h$ in every $R_j$ with 
$u+1\leqq j\leqq m+k+s$. 
Then the finitely presented group  
$G=<x_0,x_1,...,x_{k+s+\beta}\mid R_1,R_2,...,R_{m+k+s}>$ 
has the Fox derivative 
$\gamma^+(\partial R_j /\partial x_i)=b_{ij}$ for every $i, j$. 
We note that $G/[G,G]=\Z^{1+k+s+\beta-u}=\Z^{1+n}$. 
Let $\gamma_*\co G \to \Z$ be the epimorphism induced from $\gamma$. 
Then ${\rm Ker}\gamma_*/[{\rm Ker}\gamma_*,{\rm Ker}\gamma_*]\cong M.$ 
By T Yajima's construction in \cite{Ya}, there is a ribbon $F^{n+1}_g$--link 
$L$ with $\pi_1(S^4\backslash L)=G$ (hence $M(L)=M$) so that, 
in terms of a disk--arc presentation of 
a ribbon surface-link, the generators $x_i$
$(i=0,1,...,k+s+\beta)$ correspond to the oriented disks 
$D_i$ $(i=0,1,...,k+s+\beta)$, respectively, and the relation 
$R_j: w_j^{-1} x_jw_j=x_0$ (or  $w_j^{-1} x_hw_j=x_h$, respectively) 
corresponds to an oriented arc $\alpha_j$ which 
starts from a point of $\partial D_j$ 
(or $\partial D_h$, respectively), terminates at a point of 
$\partial D_0$ (or $\partial D_h$, respectively), and is described 
in the following manner: 
When $w_j$ is written as 
$x_{j_1}^{\varepsilon_1}x_{j_2}^{\varepsilon_2}\cdots$ 
($\varepsilon_i=\pm 1$),  the arc  
$\alpha_j$ should be described  
so that it first intersects  the interior of 
the disk $D_{j_1}$ in a point with sign $\varepsilon_1$.   
Next, it intersects the interior of the disk $D_{j_2}$ in a point 
with sign $\varepsilon_2$. This process should be continued 
in the order of the letters $x_{j_i}$ appearing in $w_j$ until they 
are exhausted. Thus, the arc $\alpha_j$ is constructed. 
Then we have 
\[g=m+k+s-u= m+(n-\beta)=e(E^2M)+\tau(M). \]
The arbitrariness of $h$ for the $x_h$ in $R_j$ with 
$u+1\leqq j\leqq m+k+s$ guarantees us 
to construct a 2--manifold $F^{n+1}_g=F^{n+1}_{g_1,g_2,...,g_{n+1}}$ 
corresponding to any partition $g=g_1+g_2+...+g_{n+1}$. \end{proof}

The following corollary comes directly from Lemmas \ref{2.4}, \ref{2.5} and 
\fullref{3.2}. 

\begin{Corollary}\label{3.3}
A finitely generated 
$\Lambda$--module $M$  is a cokernel-free $\Lambda$--module  
of corank  $n$ if and only if there is an $F^{n+1}_g$--link $L$ with 
$M(L)=M$  for some $g$. 
\end{Corollary}

The following corollary gives  a characterization of the modules 
$M(L)$ of ribbon  $F^{n+1}_g$--links 
$L$ with $\beta_*(L)=0$.

\begin{Corollary}\label{3.4}
A   cokernel-free 
$\Lambda$--module $M$ of corank $n$  is the module $M(L)$ 
of a ribbon $F^{n+1}_g$--link $L$ with $\beta_*(L)=0$ 
(in this case, we have necessarily $g=n$)
if and only if  $\beta(M)=0$ and $DM=0$.   
\end{Corollary}

\begin{proof}For the proof of \lq\lq if\rq\rq part, 
we note that $E^2M=E^2DM=0$ and hence $e(E^2M)+\tau(M)=n$. By \fullref{3.2}, 
we have a ribbon $F^{n+1}_n$--link $L$ with $M(L)=M$. Since $\beta(M)=0$, 
we see from \fullref{2.6} that $\beta_*(L)=0$. 
For the proof of \lq\lq only if\rq\rq part, 
we note $g=n$ by \fullref{2.6}. Hence by \fullref{3.1}, 
$n\geqq e(E^2M)+\tau(M)$.  Since $\beta(M)=0$ means $\tau(M)=n$, we have 
$e(E^2M)=0$, so that $E^2M=0$ which is equivalent to $DM=0$. 
\end{proof}

Here are two examples which are not covered by \fullref{3.4}.

\begin{Example}\label{3.5}
For a  cokernel-free 
$\Lambda$--module $M$ of corank $n$ with $\beta(M)=0$ (so that 
$\tau(M)=n$) and $DM=0$, we have the following examples (1) and (2). 

(1)\qua Let 
$M'=M\oplus\Lambda/(t+1,a)$ for an odd $a\geqq 3$. Since 
$E^2M'\cong \Lambda/(t+1,a)\ne 0$, the $\Lambda$--module $M'$ is not 
the module $M(L)$ of a ribbon $F^{n+1}_g$--link $L$ with $\beta_*(L)=0$. 
On the other hand, $\Lambda/(t+1,a)$ is wel-known to be the module of 
a non-ribbon $F^1_0$--knot $K$ (for example, the 2--twist-spun knot of 
the 2--bridge knot of type $(a,1)$) and  $M$ is the module $M(L)$ 
of a ribbon $F^{n+1}_n$--link $L$ with $\beta_*(L)=0$ by \fullref{3.4}. 
Hence $M'$ is the module $M(L')$ of a 
non-ribbon $F^{n+1}_n$--link $L'$ (taking a connected sum $L\# K$) with 
$\beta_*(L')=0$. 

(2)\qua Let $M''=M\oplus\Lambda/(2t-1,a)$ for an odd $a\geqq 5$. 
Although $M''$ is cokernel-free  of corank $n$ and $\beta(M'')=0$, 
we can show that 
$M''$ is not the module $M(L)$  of any $F^{n+1}_g$--link $L$ with 
$\beta_*(L)=0$. To see this, 
suppose $M''=M(L)$ for an $F^{n+1}_g$--link $L$. 
Since $\Lambda/(2t-1,a)$ is 
not $\Lambda$--isomorphic to $\Lambda/(2t^{-1}-1,a)=\Lambda/(t-2,a)$, 
the $\Lambda$--module $DM''=\Lambda/(2t-1,a)$ is 
not $t$--anti isomorphic to the $\Lambda$--module 
$E^2DM''={\rm hom}_\Z(DM'',\Q/\Z)\cong \Lambda/(2t-1,a)$ and hence 
by the second duality of \cite{Ka}  there is a $t$--anti isomorphism 
$$\theta\co  DM''\to E^1BH_2(\tilde E,\partial\tilde E).$$ 
This implies that $\beta_2(L)=\beta(H_2(\tilde E,\partial\tilde E))\ne 0$. 
Thus, $M''$ is not the module $M(L)$  of any 
$F^{n+1}_g$--link $L$ with $\beta_*(L)=0$. 
On the other hand, there is a ribbon 
$F^{n+1}_{n+1}$--link $L''$ with $M(L'')=M''$ by \fullref{3.2}, because  
$e(E^2M'')=e(\Lambda/(2t-1,a))=1$ and hence $e(E^2M'')+\tau(M'')=1+n$. 
In this case,  we have $\beta_2(L'')=2$ by \fullref{2.5}. 
\end{Example}

\section{A characterization of the first Alexander $\Z[\Z]$--modules 
of virtual links}\label{sec4}

\begin{figure}[htbp]
    \centering
	\includegraphics[width=.3\hsize]{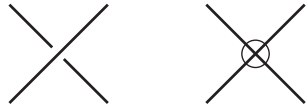}
	\caption{A real or virtual crossing point}
	\label{Realandvirtualcrossings}
\end{figure}
The notion of virtual links was introduced by L\,H Kauffman \cite{Kau}. 
A {\it virtual} $r$--{\it link diagram} is a diagram $D$  of  
immersed oriented  $r$  loops in  $S^2$  with two kinds of 
crossing points  given in \fullref{Realandvirtualcrossings}, where 
the left or right crossing point is called a 
{\it real} or {\it virtual} crossing point, respectively. 
\begin{figure}[htbp]
     \centering
	\includegraphics[width=.8\hsize]{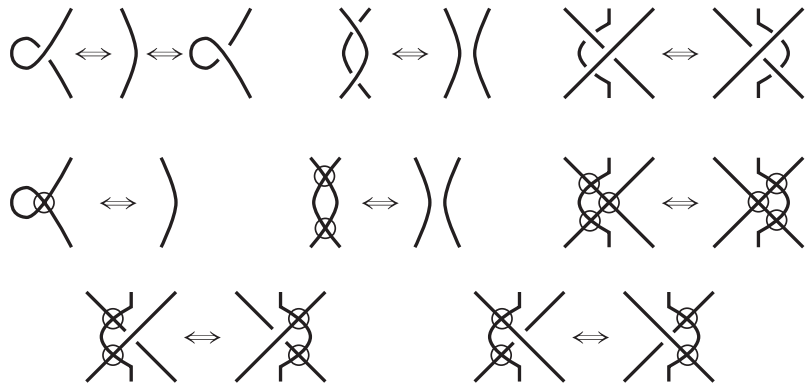}
	\caption{R-moves and Virtual R-moves}
	\label{VirtualR-moves}
\end{figure}
A  {\it virtual} $r$--{\it link} $\ell$ is the equivalence class of virtual 
$r$--link diagrams $D$ under the local moves given in 
\fullref{VirtualR-moves} which are called {\it R-moves} 
for the first three local moves and {\it virtual R-moves} 
for the other local moves. A  virtual $r$--link is called a {\it 
classical} $r$--{\it link} if it is represented by a virtual link diagram 
without virtual crossing points.
The group $\pi(\ell)$ of a virtual 
$r$--link $\ell$ is the group with Wirtinger presentation whose 
generators consist of the edges of a virtual link diagram $D$ of $\ell$   
and whose relations 
are obtained from $D$ as they are indicated in \fullref{Relations}. 
\begin{figure}[htbp]
\labellist\small
\pinlabel $a$ [l] at 63 113
\pinlabel $a$ [l] at 252 113
\pinlabel $d$ [l] at 63 10
\pinlabel $d$ [l] at 252 10
\pinlabel $b$ [b] <1.5pt,0pt> at 9 62
\pinlabel $b$ [b] <2pt, 1pt> at 193 62
\pinlabel $c$ [b] at 110 62
\pinlabel $c$ [b] <0pt, 1pt> at 306 62
\pinlabel $a=d,b=a^{-1}ca$ [t]  <0pt, -3pt> at 61 1
\pinlabel $a=d,b=c$ [t]  <0pt, -3	pt> at 252 1
\endlabellist
        \centering
	\includegraphics[width=.5\hsize]{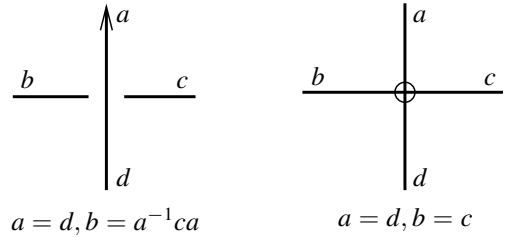}
\vspace{10pt}	
\caption{Relations}
	\label{Relations}
\end{figure}
It is easily checked 
that the Wirtinger group $\pi(\ell)$ up to Tietze equivalences 
is unchanged under  the R-moves and virtual 
R-moves. \fullref{Transforming} defines a map $\sigma'$ 
from a virtual $r$--link diagram to 
a disk--arc presentation of a ribbon  $F^r_{1,1,...,1}$--link. 
\begin{figure}[htbp]
\labellist\small
\pinlabel or at 347 59
\endlabellist
     \centering
	\includegraphics[width=.6\hsize]{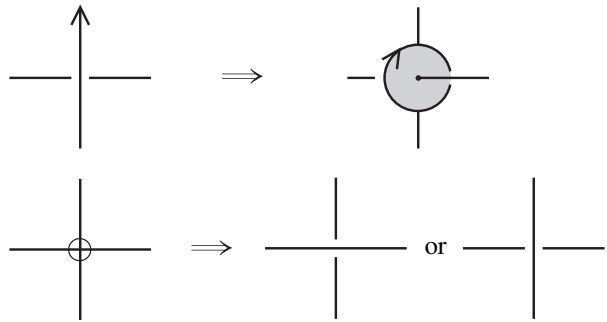}
	\caption{Definition of the map $\sigma'$}
	\label{Transforming}
\end{figure}
S Satoh proved in \cite{Sa} 
that this map $\sigma'$ induces a (non-injective) surjective map 
$\sigma$ from the set of virtual $r$--links onto the set of ribbon  
$F^r_{1,1,...,1}$--links. 
For example,  the map $\sigma$ sends  
a nontrivial virtual knot into a trivial $F^1_1$--knot in 
\fullref{TransformedTrivial}, where 
non-triviality of the virtual knot is shown by the Jones polynomial 
(see \cite{Kau}) and triviality of the  $F^1_1$--knot is shown by 
an argument of \cite{HK} on deforming a 1--handle.  
\begin{figure}[htbp]
     \centering
	\includegraphics[width=.6\hsize]{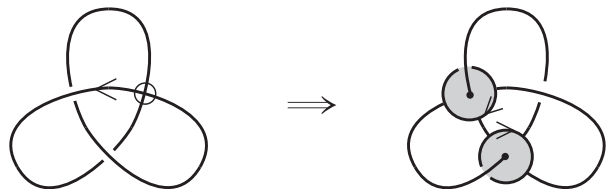}
	\caption{A non-trivial virtual knot sent to 
	the trivial $F^1_1$--knot}
	\label{TransformedTrivial}
\end{figure}
{\it It would be an important problem to find a finite number of local moves 
generating the preimage of} $\sigma$ (see \cite{Sa}). 
T Yajima in \cite{Ya} gives a Wirtinger presentation of 
the  group $\pi_1(S^4\backslash L)$ of a ribbon $F^r_g$--link $L$. 
From an analogy of the constructions,  
we  see that the map $\sigma$ induces 
the same Wirtinger presentation of  
a virtual $r$--link diagram $D$ and  the disk--arc presentation $\sigma'(D)$. 
Thus, we have  the following proposition which has been  independently 
observed by S\,G Kim \cite{Ki}, 
S Satoh \cite{Sa}, and D Silver and S Williams \cite{SW} in the case 
of virtual knots:

\begin{Proposition}\label{4.1}
The set of the groups of virtual $r$--links is the same as 
the set of the 
 groups of ribbon $F^r_{1,1,...,1}$--links. 
\end{Proposition}

For a virtual $r$--link $\ell$, let  $\gamma\co \pi(\ell) \to \Z$ 
be an epimorphism sending every generator of a Wirtinger presentation 
to 1, which is independent of a choice of Wirtinger presentations. 
The {\it first Alexander} $\Z[\Z]$--{\it module}, or simply 
the {\it module}  of a virtual $r$--link $\ell$ is the $\Lambda$--module 
$M(\ell)={\rm Ker}\gamma/[{\rm Ker}\gamma,{\rm Ker}\gamma]$. 
The following corollary comes directly from \fullref{4.1}.

\begin{Corollary}\label{4.2}
The set of the modules of virtual $r$--links 
is the same as  the set of the 
modules of ribbon $F^r_{1,1,...,1}$--links. 
\end{Corollary}

The following theorem giving  a characterization of the modules of virtual 
$r$--links comes directly from \fullref{3.2} and \fullref{4.2}.

\begin{Theorem}\label{4.3}
A finitely generated $\Lambda$--module 
$M$ is the module $M(\ell)$ of a virtual $r$--link $\ell$ 
if and only if $M$ is a  cokernel-free $\Lambda$--module of corank $r-1$   
and has $e(E^2M)\leqq 1+\beta(M)$. 
\end{Theorem}

\begin{figure}[htbp]
     \centering
	\includegraphics[clip, height=4cm]{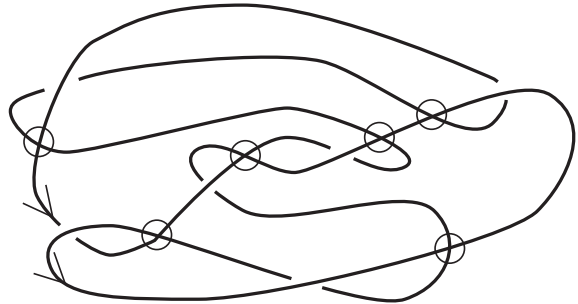}
	\caption{A virtual 2--link sent to the ribbon $F^2_{1,1}$--link in 
                 \fullref{Link}}
	\label{Virtuallink}
\end{figure}
Here is one example.

\begin{Example}\label{4.4}
The ribbon $F^2_{1,1}$--link in 
\fullref{Link} is the $\sigma$--image of a virtual 2--link $\ell$ 
illustrated in \fullref{Virtuallink} with group 
$\pi(\ell)=(x,y\mid x=(yx^{-1}y^{-1})x(yx^{-1}y^{-1})^{-1},
y=(x^{-1}yx^{-1})y(x^{-1}yx^{-1})^{-1})$ 
and module $M(\ell)=\Lambda/((t-1)^2,2(t-1))$. 
Since $DM(\ell)=\Lambda/((t-1),2)\ne 0$, the virtual $2$--link 
$\ell$ is not any classical 2--link. In fact, if $\ell$ is a classical 
link with $M(\ell)$ a torsion $\Lambda$--module, 
then we must have $DM(\ell)=0$ by the second duality of \cite{Ka} 
(cf \cite{Kaw}). It is unknown whether there is a classical link $\ell$ 
such that $t-1\co DM(\ell)\to DM(\ell)$ is not injective (cf \cite{Kaw}), but 
this example means that such a virtual link exists. 
\end{Example}

We see from \fullref{4.3} 
that $M$ is the module of a virtual knot (ie, a virtual 
1--link) if and only if $M$ is a cokernel-free $\Lambda$--module 
of corank 0 and has $e(E^2M)\leqq 1$, for we have 
$\beta(M)=0$ for every cokernel-free $\Lambda$--module of corank $0$. 
For a direct sum 
on the modules of virtual knots, we obtain the following observations.

\begin{Corollary}\label{4.5}\ 

\begin{enumerate}
\item For the  module $M$ of every virtual knot with 
$e(E^2M)=1$, the $n(>1)$--fold direct sum $M^n$ of $M$ 
is a cokernel-free $\Lambda$--module of corank 0, but not the module of 
any virtual knot. 

\item For the module $M$ of every virtual knot and  the module 
$M'$ of a virtual knot with $e(E^2M')=0$, 
the direct sum $M\oplus M'$ is the module of a virtual knot. 
\end{enumerate}
\end{Corollary}

\begin{proof} 
The module $M^n$ is obviously cokernel-free of corank 0. 
Using that $E^2M^n=(E^2M)^n$, we see that 
$e(E^2M^n)\leqq n$. If $E^2M$ has an element of a prime order $p$, then 
we consider the non-trivial $\Lambda_p$--module $(E^2M)_p= E^2M/ pE^2M$, 
where $\Lambda_p=\Z_p[\Z]=\Z_p[t,t^{-1}]$ which is a principal ideal domain. 
Using $e((E^2M)_p)=1$, we have 
$$e(E^2M^n)=e((E^2M)^n)\geqq e(((E^2M)_p)^n)=n$$ 
and hence $e(E^2M^n)=n>1$. 
By \fullref{4.3}, $M^n$ is not the module of any virtual knot, proving (1). 
For (2), the module $M\oplus M'$ is also cokernel-free of corank $0$. 
Since $E^2M'=0$, we have $E^2(M\oplus M')=E^2M$ and by \fullref{4.3} 
$M\oplus M'$ is  the module of a virtual knot, proving (2). 
\end{proof}

\section{A graded structure on the first Alexander $\Z[\Z]$--modules 
of surface-links}\label{sec5}

Let ${\cal A}^r_g$ be the set of 
the modules $M(L)$ of all $F^r_g$--links $L$, 
and ${\cal A}^r[2]=\cup_{g=0}^{+\infty}{\cal A}^r_g$. 
In this section, we show the properness of the inclusions
$${\cal A}^r_0\subset {\cal A}^r_1\subset {\cal A}^r_2\subset\cdots\subset
{\cal A}^r_n\subset\cdots\subset{\cal A}^r[2].$$  
To see this, we establish an estimate of the total genus $g$ by 
the module of a general $F^r_g$--link.  To state this estimate, 
we need some notions on a finite $\Lambda$--module. 
A finite $\Lambda$--module $D$ is {\it symmetric} if there is a $t$--anti 
isomorphism $D\cong E^2D={\rm hom}_\Z(D,\Q/\Z)$, and {\it nearly symmetric} if 
there a $\Lambda$--exact sequence 
\[0\rightarrow D_1 \rightarrow  D\rightarrow  D^* \rightarrow  D_0 
\rightarrow 0\]
such that $D_i(i=0,1)$ are  finite $\Lambda$--modules  with  
$(t-1)D_i=0$  and $D^*$ is a finite symmetric $\Lambda$--module. 
For a general $F^r_g$--link $L$, we shall show the following theorem:

\begin{Theorem}\label{5.1} If  $M$ is the module $M(L)$ of  an 
$F^r_g$--link $L$, then we have  a nearly symmetric finite 
$\Lambda$--submodule 
$D\subset DM$ such that  $g\geqq e(E^2(M/D))/2 +\tau(M)$. 
\end{Theorem}

\begin{proof} Let $F^r_g=F^r_{g_1,g_2,...,g_r}$. Let $L_i$ be 
the $F^1_{g_i}$--component of $L$, and $\partial_i E$  the component 
of the boundary $\partial E$ corresponding to $L_i$. 
We parametrize $\partial_i E$ as $L_i\times S^1$ so that 
the natural composite 
\[ H_1(L_i\times 1)\rightarrow H_1(\partial_i E)\rightarrow 
H_1(E)\overset {\gamma}{\rightarrow} \Z\]
is trivial. Let $V_i$ be the handlebody of genus $g_i$. We construct a 
closed connected oriented 4--manifold $X=E\cup(\cup_{i=1}^r V_i\times S^1)$ 
obtained by pasting $\partial_i E$ to 
$L_i\times S^1=(\partial V_i)\times S^1$. Then the infinite cyclic covering 
$p\co \tilde E\to E$ associated with $\gamma$ extends to an infinite cyclic covering $p_X\co \tilde X\to X$, so that $(p_X)^{-1}(V_i\times S^1)=V_i\times R^1$. 
Since 
$H_*(\tilde X,\tilde E)\cong \oplus_{i=1}^rH_*((V_i,\partial V_i)\times R^1)$,
the exact sequence of the pair $(\tilde X,\tilde E)$ induces 
a $\Lambda$--exact sequence 
\[ 0\rightarrow T_1\rightarrow H_1(\tilde E)\overset{i_*}{\rightarrow}  
H_1(\tilde X)\rightarrow 0\]
where $(t-1)T_1=0.$ This exact sequence induces a $\Lambda$--exact sequence 
\begin{equation}0\rightarrow D_1 \rightarrow 
DH_1(\tilde E)\overset{i^D_*}{\rightarrow}  
DH_1(\tilde X)\rightarrow D_0 \rightarrow 0\tag{5.1.1}\label{5.1.1}
\end{equation}
for some finite $\Lambda$--modules $D_i(i=0,1)$ with $(t-1)D_i=0$.

To see \eqref{5.1.1}, it suffices to prove that 
the cokernel $D_0$ of the natural homomorphism 
$i^D_*\co DH_1(\tilde E)\to  DH_1(\tilde X)$ has $(t-1)D_0=0$. 
For an element $x\in DH_1(\tilde X)$, we take an element 
$x'\in H_1(\tilde E)$ with $i_*(x')=x$. Since there is a positive integer 
$n$ such that $(t^n-1)x=0$, the element $(t^n-1)x'\in H_1(\tilde E)$ is the 
image of an element in $T_1$. Hence  $(t^n-1)(t-1)x'=0$.  Also, 
since there is a positive integer $m$ such that $mx=0$, we also see that 
$m(t-1)x'=0$, so that $(t-1)x'$ is in $DH_1(\tilde E)$ and 
$i^D_*((t-1)x')=(t-1)x$. This means $(t-1)D_0=0$, showing \eqref{5.1.1}.

By the second duality in \cite{Ka}, 
there is a natural $t$--anti epimorphism 
$\theta\co DH_1(\tilde X) \to E^1BH_2(\tilde X)$ 
whose kernel $D^*=DH_1(\tilde X)^{\theta}$ is symmetric. Then 
\[ e(E^2(DH_1(\tilde X)/D^*))=e(E^2E^1BH_2(\tilde X))\leqq 
\beta BH_2(\tilde X),\]
where the later inequality is obtained by using \fullref{2.2}. 
Since $H_*(\tilde X,\tilde E)$ is $\Lambda$--torsion, we see from 
\fullref{2.5} that  
\[\beta BH_2(\tilde X)=\beta_2(L)=2(g-\tau(L)).\]
In \eqref{5.1.1}, the $\Lambda$--submodule 
$D=(i^D_*)^{-1}(D^*)\subset DH_1(\tilde E)=DM(L)$ 
induces a $\Lambda$--exact sequence 
$0\rightarrow D_1 \rightarrow D\rightarrow  
D^*\rightarrow D'_0 \rightarrow 0$ for a finite $\Lambda$--module 
$D'_0$ with $(t-1)D'_0=0$, so that $D$ is nearly symmetric. 
Using that $i^D_*$ induces a $\Lambda$--monomorphism 
$DM(L)/D\to DH_1(\tilde X)/D^*$, we see that there is a $\Lambda$--epimorphism 
$E^2(DH_1(\tilde X)/D^*)\to E^2(DM(L)/D)$, so that 
\[e(E^2(DM(L)/D))\leqq e(E^2(DH_1(\tilde X)/D^*))\leqq 2(g-\tau(L)).\]
Thus, we have $g\geqq e(E^2(DM(L)/D))/2+\tau(L)$. \end{proof}

For an application of this theorem, it is useful to note that 
every finite $\Lambda$--module $D$ has a unique splitting $D_{t-1}\oplus D_c$ 
(see \cite[Lemma 2.7]{Ka''}), where $D_{t-1}$ is the $\Lambda$--submodule 
consisting of an element annihilated by the multiplication of 
some power of $t-1$ and 
$D_c$ is a cokernel-free $\Lambda$--submodule of corank $0$.  
As a direct consequence of this property, 
we see that if $D$ is nearly symmetric, then $D_c$ is symmetric. 
Then we can obtain the following result from \fullref{5.1}.

\begin{Corollary}\label{5.2} 
For every $r\geqq 1$, we have 
$${\cal A}^r_0 \subsetneqq {\cal A}^r_1\subsetneqq {\cal A}^r_2 
\subsetneqq {\cal A}^r_3\subsetneqq \cdots 
\subsetneqq{\cal A}^r_n\subsetneqq\cdots\subsetneqq {\cal A}^r[2]$$
and the set ${\cal A}^r[2]$ is equal to 
the set of finitely generated cokernel-free $\Lambda$--modules of corank 
$r-1$, so that ${\cal A}^r[2]\cap{\cal A}^{r'}[2]=\emptyset$ if $r\ne r'$. 
\end{Corollary}

\begin{proof}  
We have ${\cal A}^r_g \subset {\cal A}^r_{g+1}$ for every $g$ by a connected 
sum of a trivial $F^1_1$--knot. 
Let $L_0$ be a trivial $F^r_0$--link whose module $M(L_0)=\Lambda^{r-1}$. 
Let $K$ be a ribbon $F^1_1$--knot with $M(K)=\Lambda/(2t-1,k)$ for a prime 
$k\geqq 5$. This existence is given by \fullref{3.2}.  
For every positive integer $n$, let $L_n$ be an 
$F^r_n$--link obtained by a connected sum of $L_0$ and $n$ copies of $K$, and 
$M_n=\Lambda^{r-1}\oplus (\Lambda/(2t-1,k))^n$. Then we have 
$M(L_n)=M_n$. We show that if $M_n$ is the module of an $F^r_g$--link $L$, 
then $g\geqq n/2$. To see this, we note that $\tau(M_n)=0$, 
$DM_n=(\Lambda/(2t-1,k))^n=(DM_n)_c$ 
does not admit any non-trivial symmetric submodule, and 
$e(E^2M_n)=n$. Hence $g\geqq e(E^2M_n)/2 +\tau(M_n)=n/2$ by \fullref{5.1}. 
This means that 
among the modules $M_n (g+1\leqq n\leqq 2g+1)$ there is a member 
$M_n$ in ${\cal A}^r_{g+1}$ but not in ${\cal A}^r_{g}$. In fact, if 
$M_{g+1}\not\in {\cal A}^r_{g}$, then $M_{g+1}$ is a desired member. If 
$M_{g+1}\in {\cal A}^r_{g}$, then we take the largest $n(\geqq g+1)$ 
such that $M_n\in {\cal A}^r_{g}$.  
Since $M_{2g+1}\not\in {\cal A}^r_{g}$, we have 
$n<2g+1$. Let $L'$ be an $F^r_g$--link with 
$M(L')=M_n$, and $L''$ an $F^r_{g+1}$--knot which is a connected sum 
of $L'$ and $K$. Then $M_{n+1}=M(L'')$ is in ${\cal A}^r_{g+1}$ 
but not in ${\cal A}^r_{g}$. 
The characterization of ${\cal A}^r[2]$ follows directly 
from \fullref{3.3}, so that if $r\ne r'$, then 
${\cal A}^r[2]\cap{\cal A}^{r'}[2]=\emptyset$. 
\end{proof}

\section[A graded structure on the first Alexander {Z[Z]}-modules 
of classical links, surface-links and higher-dimensional 
manifold-links]{A graded structure on the first Alexander $\Z[\Z]$--modules 
of classical links, surface-links and higher-dimensional\newline 
manifold-links}\label{sec6}

An $n$--{\it dimensional  manifold-link with} $r$ {\it components} is 
the ambient isotopy class of 
a closed oriented $n$--manifold with $r$ components 
embedded in the $(n+2)$--sphere $S^{n+2}$ by a locally-flat embedding. 
A $1$--dimensional manifold-link with $r$ components coincides with 
a classical $r$--link even when we regard it as a virtual link by a result of  
M Goussarov, M Polyak and O Viro \cite{Go}. 
Let $E_Y=S^{n+2}\backslash {\rm int}N(Y)$ for a tubular neighborhood 
$N(Y)$ of $Y$ in $S^{n+2}$. Since 
$H_1(E_Y)\cong \Z^r$ has a unique oriented meridian basis, 
we have a unique infinite cyclic covering $p\co \tilde E_Y\to E_Y$ associated 
with the epimorphism $\gamma\co H_1(E_Y)\to \Z$ sending every oriented 
meridian  to $1$. The {\it first Alexander} $\Z[\Z]$--{\it module}, 
or simply the {\it module} of the manifold-link $Y$ is the $\Lambda$--module 
$M(Y)=H_1(\tilde E_Y)$. Let ${\cal A}^r[n]$ denote the set of the modules of 
$n$--dimensional manifold-links with $r$ components by generalizing 
the case $n=2$. Let $R{\cal A}^r_g$ be the set of the modules of ribbon 
$F^r_g$--links. By \fullref{3.2} and \fullref{3.3}, we have 
${\cal A}^r[2]=\cup_{g=0}^{+\infty} R{\cal A}^r_g$. 
Let $V{\cal A}^r[1]$ denote the set of the modules of virtual  
$r$--links. By \fullref{3.2} and  \fullref{4.2}, we have 
$V{\cal A}^r[1]=R{\cal A}^r_r$. 
For the set 
${\cal A}^r[1]$,  we further consider the subset 
${\cal A}^r_g[1]={\cal A}^r[1]\cap{\cal A}^r_g$. We have 
${\cal A}^r_g[1]\subset{\cal A}^r_{g+1}[1]\subset{\cal A}^r[1]$ 
for every $g\geqq 0$. 
Taking a split union of classical knots with non-trivial Alexander 
polynomials, we see that  the set ${\cal A}^r_0[1]$ is infinite. 
We have the following comparison theorem on 
the modules of classical $r$--links, $F^r_g$--links and 
higher-dimensional manifold-links with $r$ components, which explains why we 
consider the strictly nested class of classical and surface-links for 
the classification problem of the Alexander modules of general manifold-links.

\begin{Theorem}\label{6.1} 
\begin{eqnarray*}
{\cal A}^r_0[1]\subsetneqq{\cal A}^r_1[1]\subsetneqq&&\cdots\quad\subsetneqq
{\cal A}^r_{r-1}[1]={\cal A}^r[1]\subsetneqq
R{\cal A}^r_{r-1}\subsetneqq R{\cal A}^r_r=V{\cal A}^r[1]\\
&&\subsetneqq{\cal A}^r_r\subsetneqq
\cdots\subsetneqq{\cal A}^r_n\subsetneqq\cdots\subsetneqq{\cal A}^r[2]
={\cal A}^r[3]={\cal A}^r[4]=\cdots.
\end{eqnarray*}
\end{Theorem}

\begin{proof} 
By \fullref{2.4} and \fullref{3.3}, we have 
${\cal A}^r[2]\supset {\cal A}^r[n]$ for every $n\geqq 1$. 
To see that ${\cal A}^r[n]\subset {\cal A}^r[n+1]$, we use a spinning 
construction. To explain it, 
let $M(Y)\in {\cal A}^r[n]$ for a manifold-link $Y$. We choose an 
$(n+2)$--ball $B^{n+2}_o\subset S^{n+2}$ such that the pair 
$(B^{n+2}_o, Y_o)$ $(Y_o=Y\cap B^{n+2}_o)$ is homeomorphic 
to the standard disk pair $(D^2\times D^n,0\times D^n)$, 
where $D^n$ denotes the $n$--disk and 
$o$ denotes the origin of the 2--disk $D^2$.  Let 
$B^{n+2}=\text{cl}(S^{n+2}\backslash B^{n+2}_o)$ and 
$Y'=\text{cl}(Y\backslash Y_o)$. 
We construct an $(n+1)$--dimensional manifold link $Y^+\subset S^{n+3}$ by
$$ Y^+=Y'\times S^1\cup (\partial Y')\times D^2 \subset 
B^{n+2}\times S^1\cup (\partial B^{n+2})\times D^2=S^{n+3}.$$
Then the fundamental groups $\pi_1(E_Y)$ and $\pi_1(E_{Y^+})$ 
are meridian-preservingly isomorphic by van Kampen theorem and hence 
$M(Y)=M(Y^+)$. This implies that  
${\cal A}^r[1]\subset R{\cal A}^r_{r-1}$ and 
${\cal A}^r[2]={\cal A}^r[3]={\cal A}^r[4]=\cdots$.  
Let $g$ be an integer with $0<g\leqq r-1$. 
Let $\ell$ be a classical $(g+1)$--link with  
$M(\ell)$ a torsion $\Lambda$--module. 
Then $M(\ell)=M(L)$ for a ribbon $F^{g+1}_g$--link $L$ 
by the spinning construction. 
The $\Lambda$--module  $M'=M(\ell)\oplus \Lambda^{r-1-g}$ is 
in ${\cal A}^r[1]$ as the module of a split union $\ell^+$ of $\ell$ 
and a trivial $(r-1-g)$--link and in $R{\cal A}^r_g\subset {\cal A}^r_g$ 
as the module of a split union $L^+$ of $L$ and a trivial 
$F^{r-1-g}_0$--link. Hence $M'$ is in 
${\cal A}^r_g[1]$. If $M'=M(L')$ for  an $F^r_s$--link $L'$, then 
we have $\tau(L')=(r-1)-(r-1-g)=g$ and by \fullref{2.5} 
$\beta_2(L')=2(s-\tau(L'))=2(s-g)\geqq 0$. Hence $s\geqq g$. Thus, 
$M'$ is not in ${\cal A}^r_{g-1}$. This shows that 
${\cal A}^r_{g-1}[1]\subsetneqq{\cal A}^r_g[1]$ and  
$R{\cal A}^r_{g-1}\subsetneqq R{\cal A}^r_g$. This last proper inclusion 
also holds for every $g\geqq r$. In fact, by 
taking $M=(\Lambda/(t-1))^{r-1}\oplus 
(\Lambda/(t+1,a))^{g-r+1}$ for an odd $a\geqq 3$, we have 
$(E^2M)+\tau(M)=(g-r+1)+(r-1)=g.$ 
Since $M$ is cokernel-free and $cr(M)=r-1$, we have 
$M\in R{\cal A}^r_g\backslash R{\cal A}^r_{g-1}$ by \fullref{3.2}. 
Next, let $M=M(L)\in R{\cal A}^r_g$  
have $(E^2M)+\tau(M)=g$ and $p DM=0$ for an odd prime $p$. 
Let $K$ be an $S^2$--knot with $M(K)=\Lambda/(t+1,p)$
(see \fullref{3.5} (1)). Then we have $M'=M\oplus \Lambda/(t+1,p)=
M(L\# K)\in {\cal A}^r_g$ for a connected sum $L\# K$ of $L$ and $K$. 
Then we have $(E^2M')+\tau(M')=g+1$ and $M'\not\in R{\cal A}^r_g$ 
by \fullref{3.2}. Thus, $R{\cal A}^r_g\subsetneqq {\cal A}^r_g$ for every $g$. 
The properness of ${\cal A}[1]\subsetneqq R{\cal A}^r_{r-1}$ follows by a 
reason that the torsion Alexander polynomial of every classical $r$--link 
in \cite{Ka'''} is symmmetric, but there is a ribbon $S^2$--knot with 
non-symmetric Alexander polynomial (see \cite{Ka''''} for the detail). 
\end{proof}

On the inclusion ${\cal A}^r[1]\subset{\cal A}^r[2]$, we note that 
the invariant $\kappa_1(\ell)$ in \cite{Ka'''} is equal to  
the torsion-corank $\tau(L)$ 
for every  classical $r$--link $\ell$ and  every $F^r_g$--link $L$ with 
$M(\ell)=M(L)$.  

\bibliographystyle{gtart}
\bibliography{link}

\end{document}